\documentclass[10pt]{amsart}
\usepackage{amssymb,latexsym,amsmath,graphics,color}
\usepackage{amssymb,epsfig}
\usepackage {amsmath}
\usepackage {graphicx}
\usepackage {epstopdf}

\usepackage {amsthm}
\usepackage {epsf}
\usepackage{mathrsfs}
\usepackage {latexsym}
\newtheorem{theorem}{Theorem}
\newtheorem{lemma}[theorem]{Lemma}

\newtheorem{conjecture}[theorem]{Conjecture}
\newtheorem{question}[theorem]{Question}
\newcommand{\game}{\chi_{\rm{g}}}

\begin{document}
\title{The Game Chromatic Number of Trees and Forests}
\author[C.\ Dunn]{Charles Dunn}
\address{Linfield College\\
McMinnville, OR 97128}
\email{cdunn@linfield.edu}
\author[V.\ Larsen]{Victor Larsen}
\address{Emory University\\
Atlanta, GA 30322}
\email{vlarsen@emory.edu}
\author[K.\ Lindke]{Kira Lindke}
\address{Michigan Technological University\\
Houghton, MI 49931}
\email{kldurand@mtu.edu}
\author[T.\ Retter]{Troy Retter}
\address{Emory University\\
Atlanta, GA 30322}
\email{troyretter@gmail.com}
\author[D.\ Toci]{Dustin Toci}
\email{iamtoci@gmail.com}
\thanks{Partially supported by the NSF grant DMS
0649068, Linfield College}

\date{September 2010}

\begin{abstract}
While the game chromatic number of a forest is known to be at most 4, no simple criteria are known for determining the game chromatic number of a forest. We first state necessary and sufficient conditions for forests with game chromatic number 2 and then investigate the differences between forests with game chromatic number 3 and 4. In doing so, we present a minimal example of a forest with game chromatic number 4, criteria for determining the game chromatic number of a forest without vertices of degree 3, and an example of a forest with maximum degree 3 and game chromatic number 4.
\end{abstract}

\maketitle
\section{Introduction}

The map-coloring game was conceived by Steven Brams and first published in 1981 by Martin Gardner in \emph{Scientific American} \cite{G81}. However, the game was not well investigated in the mathematics community until Bodlaender reinvented it in 1991 \cite{B91} within the contexts of graphs. While the game has now been examined extensively \cite{DZ99,FKKT93,GZ99,K00,KT94,Z99,Z00}, and gone through numerous generalizations and variations \cite{CZ,CWZ00,D07,DNNPSS11,DK1,DK2,DK3,LSX99}, we now present the standard version of the original vertex coloring game.

The \emph{t-coloring game} is played between Alice and Bob on a finite graph $G$ with a set $C$ of $t$ colors. In the game, a color $\alpha \in C$ is \emph{legal} for a vertex $v$ if $v$ does not have a neighbor colored with $\alpha$.  Play begins with Alice coloring an uncolored vertex with a legal color, and progresses with Alice and Bob alternating turns coloring uncolored vertices with legal colors from $C$. If at any point there is an uncolored vertex that does not have a legal color available, Bob wins. Otherwise, Alice will win once every vertex becomes colored. The \emph{game chromatic number} of a graph $G$, denoted $\game(G)$, is the least $t$ such that Alice has a winning strategy when the $t$-coloring game is played on $G$. Clearly $\chi(G) \leq \game(G) \leq \Delta(G)+1$, where $\Delta(G)$ is the maximum degree in $G$.

In \cite{B91}, Bodlaender provided an example of a tree with game chromatic number 4 and proved that every tree has game chromatic number of at most 5. Faigle et al.\ \cite{FKKT93} were able to improve this bound.  They proved the stronger result that any forest has \emph{game coloring number} of at most 4; that is to say if the coloring game is played with an unlimited set of colors, Alice can guarantee each vertex becomes colored before 4 of its neighbors are colored. Thus, if the coloring game is played on any forest, Alice can play in a way such that every uncolored vertex has at most 3 colored neighbors, thereby guaranteeing every uncolored vertex has an available color if 4 colors are used. No previous literature has investigated the distinction between forests with different game chromatic numbers, which is the topic of this paper.

In the first two sections, we highlight the basic strategies used by Alice and Bob throughout the paper. In Section 4, we provide necessary and sufficient conditions for a forest to have game chromatic number 2.  Necessary and sufficient conditions are trivial for forests with game chromatic number 0 and 1, so the remaining part of the paper investigates the differences between forests with game chromatic number 3 and 4. In the 5th section, we present a minimal example of a forest with game chromatic number 4. Sections 6 and 7 provide criteria to determine the game chromatic number of a forest without vertices of degree 3.  Finally, Section 8 provides an example of a forest with maximum degree 3 and game chromatic number 4.

\section{Alice's Strategy: Trunks and the Modified Coloring Game (MCG)}

Suppose that $F$ is a partially colored forest. We call $R$ a \emph{trunk of $F$} if $R$ is a maximal connected subgraph of $F$ such that every colored vertex in $R$ is a leaf of $R$. An example is shown in Figure \ref{F:trunks}. Denote the set of trunks in $F$ by $\mathcal{R}(F)$. Observe that each uncolored vertex in $F$ appears in exactly one trunk in $\mathcal{R}(F)$, and each colored vertex in $F$ of degree $d$ will appear in exactly $d$ trunks of $\mathcal{R}(F)$. By this definition, every trunk contains an uncolored vertex except for the trunk composed of two colored adjacent vertices, and every vertex apart from an isolated colored vertex appears in at least one trunk. With this in mind, we will refer the unique trunk containing the uncolored vertex $u$ as $R_u$.

\begin{figure}[ht]
\leavevmode
\begin{center}
\includegraphics[width=.75\linewidth]{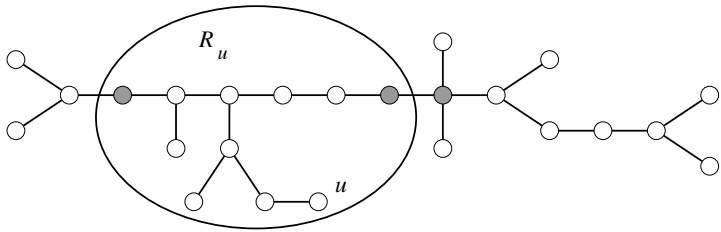}
\end{center}
\caption{An example of a partially-colored forest with six trunks, one of which is circled.}\label{F:trunks}
\end{figure}

The \emph{$k$-Modified Coloring Game ($k$-MCG)} played on a partially-colored forest $F$, is the same as the $k$-coloring game with the following exceptions:
\begin{enumerate}
\renewcommand{\labelenumi}{$\bullet$}
\setlength{\itemsep}{0pt}
\setlength{\parskip}{0pt}
\setlength{\parsep}{0pt}
\item{Bob plays first, and}
\item{Bob can choose to pass.}
\end{enumerate}

\begin{lemma} \label{l:mcg1}
Let $F$ be a partially colored forest. If Alice can win the $k$-MCG on $\mathcal{R}(F)$ she can win the $k$-coloring game on $F$.
\end{lemma}
\begin{proof}
Playing the coloring game on a partially colored forest $F$ is equivalent to playing the $k$-MCG on $\mathcal{R}(F)$ because every uncolored vertex in $F$ has the same set of neighbors in $\mathcal{R}(F)$ as in $F$.
\end{proof}

\begin{lemma} \label{l:mcg2}
Let $F$ be a partially colored forest. If Alice can win the $k$-MCG on every trunk in $\mathcal{R}(F)$ she can win the $k$-coloring game on $\mathcal{R}(F)$.
\end{lemma}
\begin{proof}
When the $k$-MCG is played on $\mathcal{R}(F)$, Alice pretends a separate game is being played on each trunk. Thus when Bob plays in a trunk, Alice responds by making a move in the same trunk using her winning strategy provided that the trunk has a remaining uncolored vertex. If Bob colors the last uncolored vertex in a trunk or passes, Alice chooses any trunk with an uncolored vertex and plays in this trunk as if Bob chose to pass.
\end{proof}

We now present the proof of Faigle et al.\ \cite{FKKT93} using the language of the MCG. 

\begin{lemma} \label{l:mcg3}
Let $F$ be a partially colored forest and suppose Alice and Bob are playing the $4$-MCG on  $\mathcal{R}(F)$. If every trunk in $\mathcal{R}(F)$ has at most 2 colored vertices, Alice can win the $4$-MCG on  $\mathcal{R}(F)$.
\end{lemma}
\begin{proof}
We proceed by induction on the number $n$ of uncolored vertices in $\mathcal{R}(F)$ and note Alice has won when $n=0$. If $n>0$, observe that after Bob's first move there can be at most one trunk with 3 colored vertices. If such a trunk exists, the trunk is a tree that has 3 colored leaves. In this tree there must be an uncolored vertex of degree at least 3 whose deletion will pairwise disconnect the leaves. Since this vertex has at most 3 colored neighbors, Alice can color this vertex with some available color. This creates a partially colored graph whose trunks each have at most 2 uncolored vertices, placing Alice in a winning position. If no trunk has 3 colored vertices and Alice has not already won, Alice chooses to play in any trunk that has an uncolored vertex. If this trunk has at most 1 colored vertex, Alice can color any vertex to reach a winning position. If this trunk has two colored vertices $x$ and $y$, Alice can color any vertex on the unique $x,y$-path to reach a winning position.
\end{proof}

\begin{theorem}[Faigle et al.]\label{t:mcg1}
For any Forest $F$,  $\game (F) \leq 4$.
\end{theorem}
\begin{proof}
Every trunk in $\mathcal{R}(F)$ initially has no colored vertices, so Alice can win the $4$-MCG on $\mathcal{R}(F)$ by Lemma \ref{l:mcg3}. By Lemma \ref{l:mcg1}, this implies Alice can win the 4-coloring game on $F$.
\end{proof}

\section{Bob's Strategy: The Expanded Coloring Game (ECG), Winning Subgraphs, and Winning Moves}

The \emph{$k$-Expanded Coloring Game} (\emph{$k$-ECG}) played on a partially colored forest $F$ is the same as the $k$-coloring game with the following exceptions:
\begin{enumerate}
\renewcommand{\labelenumi}{$\bullet$}
\setlength{\itemsep}{0pt}
\setlength{\parskip}{0pt}
\setlength{\parsep}{0pt}
\item{Alice can choose not to color a vertex on her turn, and}
\item{If Alice chooses not to color a vertex on her turn, she can choose to add a single colored leaf to $F$.}
\end{enumerate}

\begin{lemma} \label{l:ecg}
Let $F$ be a partially colored forest and let $F'$ be an induced connected subgraph of $F$. If for every $v\in V(F')$, $v$ has no colored neighbors in $V(F) \setminus V(F')$ and Bob can win the $k$-ECG on $F'$, then Bob can win the $k$-coloring game on $F$.
\end{lemma}
\begin{proof}
When the $k$-coloring game is played on $F$, Bob pretends that the $k$-ECG is being played on $F'$. When Alice colors a vertex in $N(F')\setminus F'$, Bob pretends Alice added an appropriately colored leaf to $F'$. Similarly when Alice colors a vertex not adjacent to any vertex in $V(F')$, Bob plays in $F'$ as if Alice has just passed.
\end{proof}

When Bob is playing the $k$-ECG or $k$-coloring game, we call coloring a vertex $v$ with $\alpha$ a \emph{winning move} if $v$ is uncolored, $\alpha$ is legal for $v$, and coloring $v$ with $\alpha$ will make some vertex uncolorable. Clearly Bob can win if it is his turn and there is a winning move. If coloring a vertex $v$ with $\alpha$ is a winning move and it is Alice's turn, Bob can win on his next turn unless Alice colors $v$ or colors a neighbor of $v$ with $\alpha$. We refer to two winning moves as \emph{disjoint} if the distance between the two target vertices is more than 2. If a forest has two disjoint winning moves on Alice's turn, Bob can win the $k$-ECG or $k$-coloring game. Hence when Bob wins the $k$-ECG on $F'$, there will be an uncolorable vertex in $F$.

\section{Forests with Game Chromatic Number 2}

We now classify all trees and forests with game chromatic number 2.  What is most interesting here is that the characteristics are actually different for the two classes of graphs.

\begin{lemma} \label{l:gcnp4}
Bob can win the $2$-ECG on an uncolored $P_5 := v_1,v_2,v_3,v_4,v_5$.
\end{lemma}
\begin{proof}
If Alice colors a vertex of the $P_5$ or adds a colored leaf to the $P_5$ of some color $\alpha$, Bob can color a vertex at distance two away with a different color $\beta$. At this point, the vertex between the two colored vertices has no legal color available. Otherwise if Alice chooses to pass, Bob colors $v_3$ with $\alpha$. Coloring $v_1$ with $\beta$ and coloring $v_5$ with $\beta$ are now disjoint winning moves.
\end{proof}

\begin{lemma} \label{l:gcnp3}
Bob can win the $2$-ECG on an uncolored $P_4^+$, shown in Figure \ref{F:Bob}.
\end{lemma}

\begin{figure}[ht]
\leavevmode
\begin{center}
\includegraphics{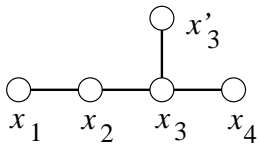}
\end{center}
\caption{The graph $P_4^+$.}\label{F:Bob}
\end{figure}

\begin{proof}
If Alice colors a vertex of the $P_4^+$ or adds a colored leaf to the $P_4^+$ of some color $\alpha$, Bob can win by coloring a vertex at distance two away with a different color $\beta$. Otherwise if Alice chooses to pass, Bob colors $x_4$ with $\alpha$. Since coloring $x_2$ or $x_3'$ with $\beta$ are both winning moves, Bob can immediately win if Alice does not color $x_3$ with $\beta$. If Alice colors $x_3$ with $\beta$, coloring $x_1$ with $\alpha$ is a winning move.
\end{proof}

\begin{theorem}\label{t:gcn2}
Let $F$ be a forest and let $\ell (F)$ be the length of the longest path in $F$. Then $\game(F)=2$ if and only if:
\begin{enumerate}
\renewcommand{\labelenumi}{$\bullet$}
\setlength{\itemsep}{0pt}
\setlength{\parskip}{0pt}
\setlength{\parsep}{0pt}
\item{$1 \leq \ell (F) \leq 2$ or}
\item{$\ell (F) = 3$, $|V(F)|$ is odd, and every component with diameter $3$ is a path.}
\end{enumerate}
\end{theorem}
\begin{proof}

Suppose the condition in the theorem is not met by some forest $F$. Clearly $\ell(F) < 1$ or $\ell(F)\geq 3$. If $\ell(F) < 1$, it is easy to see $\game(F) \leq 1$. If $\ell(F)\geq 3$, we consider two cases. If $\ell(F) > 3$, $P_5 \subseteq F$ and Bob can win the $2$-ECG on the $P_5$ by Lemma \ref{l:gcnp4}, which implies $\game(F) >2 $ by Lemma  \ref{l:ecg}. If $\ell(F) = 3$ and there exists some $P_4^+ \subseteq F$, Bob can win the $2$-ECG on the $P_4^+$ by Lemma \ref{l:gcnp3}, showing $\game(F) >2 $ by Lemma  \ref{l:ecg}. Otherwise if $\ell(F) = 3$ and there does not exist a $P_4^+ \subseteq F$, we conclude $|V(F)|$ is even and some component of $F$ is a $P_4$. Since $|V(F)|$ if even in this case, Bob can guarantee that he will win the game or that Alice will be the first to play in the $P_4$. If Alice is the first to play in the $P_4$, Bob can then color a second carefully chosen vertex in the $P_4$ to make the $P_4$ uncolorable.

If a forest $F$ satisfies the condition of the theorem, the initial trunks in $\mathcal{R}(F)$ are the components of $F$. Each trunk $F'$ is either a $P_4$ or has $\ell(F') \leq 2$. It is trivial to show Alice can win the $2$-MCG on any trunk that is not a $P_4$.  Thus Alice can win the 2-MCG on the subgraph $G\subset F$ obtained by deleting all copies of $P_4$ from $F$.  
If $G=F$, then Alice can win the 2-coloring game on $F$ by Lemma \ref{l:mcg1} and Lemma \ref{l:mcg2}.  Otherwise, there is at least one $P_4$ in $F$.  In this case, when playing the 2-coloring game on $F$, Alice will follow her 2-MCG strategy on $G$ for her first turn and every turn after Bob colors a vertex in $V(G)$.  But whenever Bob colors a vertex in a $P_4$ Alice will immediately color a vertex at distance 2 using the same color.  Because $|V(G)|$ is odd, Bob cannot force Alice to play first in a $P_4$.  Thus Alice will win the 2-coloring game on $F$ with this strategy.
\end{proof}

\section{Minimal Order Forest with Game Chromatic Number 4}

While Alice can always win the 4-coloring game on every forest $F$, we have shown that 4 colors are not always necessary.  So we now wish to find a forest with the least number of vertices that requires 4 colors for Alice to be victorious.  We begin by proving a number of lemmas.

\begin{lemma} \label{l:size}
If $T$ is a tree with $|V(T)| \leq 13$, there exists a vertex $v \in V(T)$ such that every component of $T-v$ has at most 6 vertices.
\end{lemma}
\begin{proof}
Choose $v$ such that the order of the largest component in $T-v$ is minimized.  Let $X$ be the set of vertices in a maximal order component of $T-v$ and let $\overline{X} = V(T) \setminus X$. Suppose $|X|>6$. Then $| \overline{X}| \leq 13-7 =6$. Let $v'$ be the neighbor of $v$ in $X$; note that in $T-v'$, there are no edges between vertices in $X$ and $ \overline{X}$. Furthermore, there are at most 6 vertices in $V(T-v') \cap \overline{X}$ and at most $|X|-1$ vertices in $V(T-v')\cap X$. Thus the order of the largest component in $T-v'$ is $\max\{6,|X|-1\}$, which is less than $|X|$.  This contradicts our choice of $v$, so we conclude that $|X| \leq 6$.
\end{proof}

\begin{lemma}\label{l:d3}
If $T$ is a tree on at most 7 vertices, $T$ has at most 2 vertices of degree greater than 2.
\end{lemma}
\begin{proof}
Let $S$ be the set of vertices of degree greater than 2.  Observe that there can be at most $|S|-1$ edges between vertices in $S$. We now generate a lower bound for the number of edges in $T$ by counting the number of edges incident to vertices in $S$ and then subtracting the maximum number of edges that could have been counted twice: $$|E(T)| \geq 3|S|-(|S|-1) = 2|S|+1.$$ Thus $7 \geq |V(T)| = |E(T)|+1 \geq 2|S|+2$, which implies that $|S|$ is at most 2.
\end{proof}

\begin{lemma}\label{l:d4}
If $T$ is a tree on at most 7 vertices, $T$ has at most one vertex of degree greater than 3.
\end{lemma}
\begin{proof}
Assume otherwise. Then there are at least two vertices of degree at least 4 or more and there can be at most one edge between these two vertices. Therefore, $$|V(T)| = |E(T)|+1 \geq (4 \cdot 2-1)+1=8,$$ yielding a contradiction.
\end{proof}

\begin{lemma}\label{l:trunks}
If $R$ is a trunk with at most one colored vertex and order at most 7, Alice can win the 3-MCG on $R$.
\end{lemma}

\begin{proof}
We refer to a vertex $v$ as \emph{dangerous} if $v$ has at least as many uncolored neighbors as legal colors available for it. Alice can easily win once no dangerous vertices remain or when the graph has a single dangerous vertex with at most one colored neighbor. If there are fewer than 2 dangerous vertices after Bob's first turn, Alice colors any remaining dangerous vertex in her next move to win. By Lemma \ref{l:d3} there can be at most two dangerous vertices in $R$, since a dangerous vertex must have degree at least 3. We now consider the remaining cases where after Bob's first turn there are two dangerous vertices:

\emph{Case 1: Some dangerous vertex $v$ has no colored neighbors after Bob's first turn}. Alice colors the other dangerous vertex with any legal color to reach a winning position.

\emph{Case 2: Some dangerous vertex $v$ has two colored neighbors after Bob's first turn}. Alice colors $v$ with any available color. Since the three colored vertices now lie on a $P_3$, the other dangerous vertex $v'$ can have at most one colored neighbor and Alice has reached a winning position.

\emph{Case 3: Each dangerous vertex has exactly one colored neighbor after Bob's first turn}. By Lemma \ref{l:d4}, we conclude that $d(x)=3$ for one of the dangerous vertices $x$; we refer to the second dangerous vertex as $y$. If $x$ is not adjacent to $y$, Alice colors $x$ with any available color to win. Otherwise, there must be some neighbor $u$ of $x$ that is uncolored, not adjacent to $y$, and not adjacent to any previously colored vertex. Alice colors $u$ with the same color as the previously colored neighbor of $x$. At this point $x$ is no longer dangerous and $y$ has one colored neighbor, so Alice has reached a winning positon.
\end{proof}

\begin{theorem}\label{t:u13}
Let $F$ be a forest such that $|V(F)| \leq 13$. Then $\game(F)\leq 3$.
\end{theorem}
\begin{proof}
Let $T$ be a largest order component $F$. By Lemma \ref{l:size} there exists a vertex $v$ such that every component of $T-v$ has size at most 6.  Alice colors $v$, at which point every trunk in $\mathcal{R}(F)$ has order at most 7 and at most one colored vertex. By Lemma \ref{l:trunks}, Alice can win the 3-MCG on every trunk of $\mathcal{R}(F)$. Therefore, Alice will win the 3-coloring game on $F$ by Lemma \ref{l:mcg2}.
\end{proof}

\begin{lemma}[]\label{l:pcl2}
Let $T$ be the partially colored tree shown in Figure \ref{F:Pa3} that may have an additional leaf $v$ not shown in the figure colored $\beta$. Then Bob can win the $3$-ECG on $T$.
\end{lemma}

\begin{figure}[ht]
\leavevmode
\begin{center}
\includegraphics{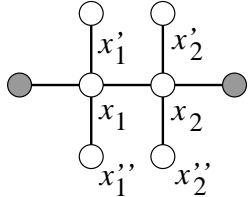}
\end{center}
\caption{A partially colored tree; shaded vertices are colored $\alpha$. The tree $T$ may have an additional leaf $v$ colored $\beta$ that is not shown.}\label{F:Pa3}
\end{figure}

\begin{proof}
If the additional leaf $v$ is adjacent to $x_1$ or $x_2$, it is not difficult to verify Bob can win on his next turn. Otherwise, we assume that $v$ does not exist or is adjacent to either $x_1'$, $x_1''$, $x_2'$, or $x_2''$. 

If Alice colors $x_1$ or $x_2$, or if Alice adds a colored leaf to $x_1$ or $x_2$ that is not colored $\alpha$, Bob has a winning move. If Alice adds a colored leaf to $x_1$ or $x_2$ that is colored $\alpha$, Bob may pretend that the leaf does not exist, since it will not affect the legal colors available for any vertex. We now consider the remaining cases where Alice passes, or where the vertex colored or leaf added by Alice is in the region $S$ of the graph, as shown in Figure \ref{F:Pa4}. Also without loss of generality we assume the initial colored leaf $v$ is not present or is in one of the 3 depicted positions.

\begin{figure}[ht]
\leavevmode
\begin{center}
\includegraphics{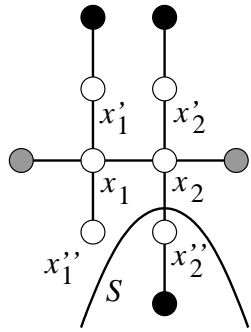}
\end{center}
\caption{Shaded vertices are colored $\alpha$. Region $S$ where Alice might have played is circled, and solid vertices denote possible location for the initial colored vertex $v$.}\label{F:Pa4}
\end{figure}

If $\beta$ is not a legal color for $x_2$, Bob can color $x_1$ with $\gamma$ to win. Otherwise, Bob colors $x_1'$ with $\gamma$. Coloring $x_1''$ or $x_2$ with $\beta$ are now winning moves, so Bob can win unless Alice responds by coloring $x_1$ with $\beta$. If Alice responds by coloring $x_1$ with $\beta$, Bob colors $x_2'$ with $\gamma$ to win.
\end{proof}

\begin{theorem} \label{l:14v}
Let $T'$ be the tree on 14 vertices shown in Figure \ref{F:min4}. Then Bob can win the 3-ECG on $T'$.  Moreover, $T'$ is a minimal example of a tree with game chromatic number 4.
\end{theorem}

\begin{proof}
If Alice colors any vertex in $T'$, Bob colors a vertex at distance three away with the same color to reach a winning position by Lemma \ref{l:pcl2}. Likewise, Bob can immediately reach a winning position if Alice adds a colored leaf adjacent to $x_1,x_2,x_3,$ or $x_4$. Otherwise, without loss of generality we assume that the colored leaf is a distance two away from $x_1$ or $x_2$ or that no colored leaf exists.  One can easily verify by considering cases and using Lemma \ref{l:pcl2} that Bob can win.

To see that $T'$ is a minimal example of a tree with maximum game chromatic number, we simply note that by Theorem \ref{t:u13}, any forest with fewer vertices has game chromatic number at most 3.
\end{proof}

\begin{figure}[ht]
\leavevmode
\begin{center}
\includegraphics{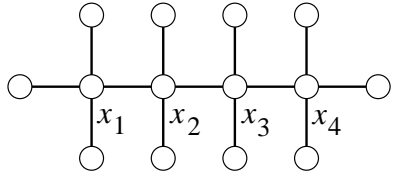}
\end{center}
\caption{The graph $T'$.}\label{F:min4}
\end{figure}

In fact, by considering different cases of maximum degree and diameter, it can be shown that $T'$ is the unique forest on 14 vertices with game chromatic number 4.  In Section 8 we will consider an example of a tree $T$ with $\Delta(T)=3$ and $\game(T)=4$.  Hence, the existence of the tree in Figure \ref{F:Pa3} as a subgraph is not necessary to force the maximum game chromatic number in a tree.

\section{Forests with Game Chromatic Number 3 and the $k$-Reduced Coloring Game ($k$-RCG)}

In this section we provide a class of forests with game chromatic number at most 3.  We prove this by showing that Alice can win the 3-coloring game on any forest in the class. In the next section we show that if $F$ is not a member of this class, has no vertices of degree 3, and $\game(F)\ge 3$, then $\game(F)=4$. Combined with the necessary and sufficient conditions for forests with game chromatic numbers 0, 1, and 2, this completely classifies forests without vertices of degree 3.

For a forest $F$, it is useful to consider the types of edges in the graph. We let $E_{>2}(F) := \{xy \in E(F): d(x)>2 \text{ or } d(y) > 2 \}$ and $E_{>>2}(F) := \{xy \in E(F): d(x)>2 \text{ and } d(y) > 2 \}$. Thus $E_{>>2}(F) \subseteq E_{>2}(F)$. If $F$ is partially colored, we define the \emph{Reduced Graph of $F$}, denoted $\mathcal{R}'(F)$, to be $\mathcal{R}(F)-\{xy: xy \not \in E_{>2}(\mathcal{R}(F))  \}$. It is important to note that $\mathcal{R'}(F)$ is generated by first looking at $\mathcal{R}(F)$ and then deleting edges based upon the degrees of the vertices in $\mathcal{R}(F)$, not the degrees of the vertices in $F$. The order is significant because a colored vertex may have high degree in $F$ but will have degree at most 1 in $\mathcal{R}(F)$. Thus two colored vertices will never be adjacent in $\mathcal{R'}(F)$, and each edge in $\mathcal{R'}(F)$ has at least one endpoint of degree greater than 2 in $\mathcal{R'}(F)$.

The \emph{$k$-Reduced Coloring Game ($k$-RCG)} is the same as the $k$-coloring game except:
\begin{enumerate}
\renewcommand{\labelenumi}{$\bullet$}
\setlength{\itemsep}{0pt}
\setlength{\parskip}{0pt}
\setlength{\parsep}{0pt}
\item{Alice can only color vertices of degree at least $k$,}
\item{Alice wins once every vertex of degree at least $k$ becomes colored,}
\item{Bob plays first, and}
\item{Bob can choose to pass.}
\end{enumerate}

\begin{lemma}[]\label{l:f'}
If Alice can win the $3$-RCG on $\mathcal{R}'(F)$ then Alice can win the 3-coloring game on $F$.
\end{lemma}
\begin{proof}
When the 3-coloring game is played on $F$, Alice pretends the 3-RCG is being played on $\mathcal{R'}(F)$.  Any move she makes in $\mathcal{R'}(F)$ will correspond to a legal move in the 3-coloring game on $F$ since every uncolored vertex of degree at least 3 has the same set of neighbors in $\mathcal{R'}(F)$ as in $F$. Any move Bob makes in the 3-coloring game on $F$ will likewise correspond to a legal move in the 3-RCG on $\mathcal{R'}(F)$ since $E(\mathcal{R}'(F)) \subseteq E(F)$. Hence, Alice can guarantee every uncolored vertex of degree at least 3 in $\mathcal{R'}(F)$ will become colored, at which point every uncolored vertex in $F$ of degree at least 3 will be colored as well. Every remaining uncolored vertex of degree less than 3 will always have a legal color available, so Alice can now abandon her strategy and color greedily on $F$ to win.
\end{proof}

\begin{lemma}\label{l:ra1}
Let $F$ be a partially colored forest. Alice can win the 3-RCG on $\mathcal{R}'(F)$ if for every trunk $R'$ in $\mathcal{R}'(F)$:
\begin{enumerate}
\renewcommand{\labelenumi}{\roman{enumi})}
\setlength{\itemsep}{0pt}
\setlength{\parskip}{0pt}
\setlength{\parsep}{0pt}
\item $R'$ has one colored vertex and does not have an edge in  $E_{>>2}(\mathcal{R}'(F))$ OR
\item $R'$ has no colored vertices and there exists a vertex $v \in V(R')$ that covers $E_{>>2}(R')$.
\end{enumerate}
\end{lemma}
\begin{proof}
We proceed by induction on the number $n$ of uncolored vertices of degree at least 3 in $\mathcal{R}'(F)$ and observe Alice has won when $n=0$.  When $n \geq 1$, we consider the following outcomes of Bob's first turn:

\emph{Case 1: Bob passes}. If there exists a trunk with no colored vertices, Alice colors the covering vertex $v$ and reaches a winning position in the new reduced graph. Otherwise there must be a trunk $R'$ with one colored vertex $x$ and an uncolored neighbor $x'$ of degree at least 3. Alice colors $x'$ and the new reduced graph is winning.

\emph{Case 2: Bob colors a vertex $b$ in a trunk $R'$ that does not have an edge in $E_{>>2}(\mathcal{R}(F))$}. Observe that every trunk in the reduced graph will now meet the inductive hypothesis except for possibly one trunk that has no edges in $E_{>>2}(\mathcal{R}(F))$ and two colored vertices $x$ and $y$. If such a trunk exists, the neighbor of $x$ must be an uncolored vertex of degree at least 3; call it $x'$. Alice colors $x'$ with any available color. Because each neighbor of $x'$ had degree 1 or 2, $x'$ will be an isolated vertex in the new reduced graph, and $x$ and $y$ will no longer be in the same trunk. The new reduced graph is now winning.

\emph{Case 3: Bob colors a vertex $b$ in a uncolored trunk $R'$ and there exists some $v \in V(R')$ such that $v$ covers $E_{>>2}(R')$}. Let $d(b,v)$ denote the distance between $b$ and $v$. If $b=v$, the resulting reduced graph is winning. If $d(b,v)=1$, Alice colors $v$ and the reduced graph is winning. If $d(b,v)=2$, Alice colors $v$ with the same color as $b$. In the reduced graph, every trunk fulfills the inductive hypothesis except for one trunk that has two leaves of the same color that are both adjacent to the same vertex. However, Alice can pretend that one of these leaves is not present, which will not change the legality of any possible play. Thus, this position is equivalent to a winning position. Finally if $d(b,v)>2$, observe that the neighbor of $b$ on the $b,v$-path, $b'$, is an uncolored vertex of degree at least 3, as each edge in a reduced graph must have an endpoint of degree at least 3. Further, because $d(b',v)\geq2$, the neighbors of $b'$ must have degree 1 or 2. So Alice colors $b'$ and in the resulting reduced graph, $b'$ will be an isolated vertex and $b$ and $v$ will no longer be in the same trunk. Hence, the resulting reduced graph will be winning.
\end{proof}

\begin{theorem}\label{t:ra}
A forest $F$ has game chromatic number of at most 3 if there exists a vertex $b$ such that once $b$ becomes colored, every reduced trunk $R'$ in $\mathcal{R'}(F)$ either
\begin{enumerate}
\renewcommand{\labelenumi}{\roman{enumi})}
\setlength{\itemsep}{0pt}
\setlength{\parskip}{0pt}
\setlength{\parsep}{0pt}
\item has one colored vertex and does not have an edge in  $E_{>>2}(\mathcal{R}'(F))$ OR
\item has no colored vertices and there exists a vertex $v \in V(R')$ that covers $E_{>>2}(R')$.
\end{enumerate}
\end{theorem}

\begin{proof}
Alice colors $b$. Every reduced trunk is now winning by  Lemma \ref{l:ra1}.
\end{proof}

Determining if the vertex $b$ required in Theorem \ref{t:ra} exists for a uncolored forest $F$ is simple. If no trunk (i.e. component) in the uncolored graph $\mathcal{R}'(F)$ contains disjoint edges in $E_{>>2}(\mathcal{R}'(F))$, $b$ will exist. If two or more trunks in $\mathcal{R}'(F)$ each contain disjoint edges in $E_{>>2}(\mathcal{R}'(F))$, $b$ will not exist. If only one trunk contains disjoint edges in $E_{>>2}(\mathcal{R}'(F))$, $b$ must lie on the path between the edges if it is to exist. Thus, 3 pairwise disjoint edges from $E_{>>2}(\mathcal{R}'(F))$ in a trunk of $\mathcal{R}'(F)$ will leave at most one possible candidate for $b$ that must be checked.

Although Theorem \ref{t:ra} provides sufficient conditions for a forest to have game chromatic number at most 3, it is apparent that there are many forests that do not meet the conditions of Theorem \ref{t:ra} but still have game chromatic number 3. For instance, the tree of 12 vertices that can by drawn by adding on leaves to a $P_5$ until every vertex on the $P_5$ has degree 3. In the next section, we show that these conditions are both necessary and sufficient for a forest to have game chromatic number at most 3 if the forest does not have vertices of degree exactly 3.

\section{Forests with Game Chromatic Number 4}

In any partial coloring of a forest $F$ with the color set $C$, there is an implicit coloring function $c: V' \to C$, where $V'$ is the set of colored vertices in $V(F)$. We thus refer to the color of a colored vertex $v$ as $c(v)$. Furthermore, when we suppose Alice and Bob are playing the 3-coloring game on a partially colored graph, we assume that the graph has been partially colored with the same set of 3 colors to be used in the game. For any two vertices $x$ and $y$ in the same trunk, we will denote the unique $x,y$-path by $P_{x,y}$.

\begin{lemma}\label{l:esl}
Let $R$ be a partially colored trunk with two colored vertices $x$ and $y$ and possibly a third colored vertex $v$ such that:
\begin{enumerate}
\renewcommand{\labelenumi}{\roman{enumi})}
\setlength{\itemsep}{0pt}
\setlength{\parskip}{0pt}
\setlength{\parsep}{0pt}
\item $c(x) \neq c(y)$,
\item $d(x,y)$ is even, and
\item every vertex an odd distance from $x$ on $P_{x,y}$ has degree at least 4,
\end{enumerate}
\noindent then Bob can win the 3-ECG on $R$ if Alice chooses to pass on her first turn.
\end{lemma}

\begin{figure}[ht]
\leavevmode
\begin{center}
\includegraphics{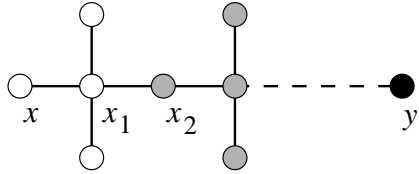}
\end{center}
\caption{Diagram showing part of $R$, which may have an additional colored leaf (not shown). Light gray represents vertices that may not exist.}\label{F:Pa6}
\end{figure}

\begin{proof}
Label the vertices as in Figure \ref{F:Pa6}. Observe that the third vertex $v$ can not lie on $P_{x,y}$ since $R$ is a trunk. Let $x_n$ be the vertex a distance $n$ from $x$ on $P_{x,y}$. If the length of $P_{x,y}$ is 2, Bob can immediately color an uncolored neighbor of $x_1$ with a color different than $c(x)$ and $c(y)$, making $x_1$ uncolorable.  If the length of $P_{x,y}$ is greater than 2, Bob can color $x_2$ with a color other than $c(x)$ or $c(y)$, unless this is prevented by $v$. This would create two new trunks, one of which will be winning by induction on Bob's next turn.  Now suppose that the color and placement of $v$ prevents coloring $x_2$ as desired. In this case $c(v),c(x),c(y)$ are all distinct colors.  Then Bob will use $c(v)$ to color an unlabelled neighbor of $x_1$; that is, a neighbor which isn't $x_2$.  Coloring $x_2$ or another unlabelled neighbor of $x_1$ with $c(y)$ are now winning moves, so Bob can win unless Alice responds by coloring $x_1$ with $c(y)$.  If Alice responds by coloring $x_1$ with $c(y)$, Bob colors $x_3$ with $c(x)$ to win.
\end{proof}

\begin{theorem}\label{ug1}
A forest $F$ that does not have vertices of degree 3 has game chromatic number 4 if coloring any vertex $b$ with $\alpha$ produces a graph where some reduced trunk $R'$ in $\mathcal{R'}(F)$ either
\begin{enumerate}
\renewcommand{\labelenumi}{\roman{enumi})}
\setlength{\itemsep}{0pt}
\setlength{\parskip}{0pt}
\setlength{\parsep}{0pt}
\item has one colored vertex and an edge in  $E_{>>2}(\mathcal{R}'(F))$ OR
\item has no colored vertices and there does not exists a vertex $v \in V(R')$ that covers $E_{>>2}(R')$.
\end{enumerate}
\end{theorem}
\begin{proof}

\emph{Case 1:} There exists a reduced trunk $R'$ in  $\mathcal{R'}(F)$ that has one colored vertex $v$ and has an edge in  $E_{>>2}(\mathcal{R}'(F))$. Choose $xy$ to be the edge in $E_{>>2}(\mathcal{R}'(F))$ such that $d(v,x)$ is minimal. If $v$ is adjacent to $x$, Bob can win by Lemma \ref{l:pcl2}. Otherwise, observe that two adjacent vertices of degree at least 4 on $P_{v,x}$ would contradict our choice of $x$ and two adjacent vertices of degree less than 2 would contradict $x$ and $v$ being connected in $\mathcal{R}'(F)$. Thus we can conclude that every vertex an odd distance from $x$ on $P_{v,x}$ has degree at least 4 and every vertex an even distance from $x$ on $P_{v,x}$ has degree exactly 2. 

\begin{figure}[ht]
\leavevmode
\begin{center}
\includegraphics{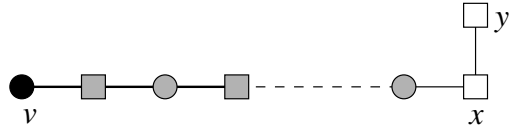}
\end{center}
\caption{Diagram showing part of the trunk $R'$. Squares denote vertices of degree at least 4 and circles denote vertices of degree exactly 2.}\label{F:Pa7}
\end{figure}

Bob now colors $x_1$, the neighbor of $x$ on $P_{v,x}$, with a color other than $c(v)$. This creates two trunks. If Alice next move is not in $R_x$, Bob can color a neighbor of $y$ and will win by Lemma \ref{l:pcl2}. Otherwise, Alice will not have colored in the other trunk $R_v$ and Bob can win by Lemma \ref{l:esl}.

\emph{Case 2:} If there exists a reduced trunk $R'$ in  $\mathcal{R'}(F)$ that has no colored vertices and there does not exists a vertex $v \in V(R')$ that covers $E_{>>2}(R')$, then $R'$ has two disjoint edges in  $E_{>>2}(\mathcal{R}'(F))$. Choose $xx'$ and $yy'$ to be disjoint edges in $E_{>>2}(\mathcal{R}'(F))$ such that $d(x,y)$ is minimal. If $d(x,y)=1$, Bob can win the 3-ECG on $F$ by Theorem \ref{l:14v} and thus the 3-coloring game by Lemma \ref{l:ecg}. If $d(x,y)>1$, observe that two adjacent vertices of degree at least 4 on $P_{x,y}$ would contradict our choice of $x$ and $y$ and two adjacent vertices of degree less than 2 would contradict $x$ and $y$ being connected in $\mathcal{R}'(F)$. Thus we can conclude that every vertex an even distance from $x$ on $P_{x,y}$ has degree at least 4 and every vertex an odd distance from $x$ on $P_{x,y}$ has degree exactly 2. 

\begin{figure}[ht]
\leavevmode
\begin{center}
\includegraphics{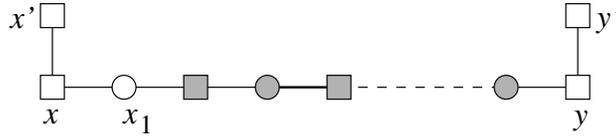}
\end{center}
\caption{Diagram showing an induced subgraph of $F$. Squares denote vertices of degree at least 4 and circles vertices of degree exactly 2.}\label{F:Pa8}
\end{figure}

Bob now colors the neighbor of $x$ on $P_{x,y}$, $x_1$, with $\beta$. This creates two new trunks, $R_x$ and $R_y$, each of which can have at most one colored leaf other than $x_1$. If Alice does not respond by coloring in $R_x$, Bob's can win on his next turn by Lemma \ref{l:pcl2}. If Alice colors in $R_x$ (and not in $R_y$), Bob can win in $R_y$ by the argument presented in Case 1.
\end{proof}

\begin{theorem}\label{f2}
A forest $F$ that does not have vertices of degree 3 has game chromatic number less than 4 if and only if there exists a vertex $a$ such that once $a$ becomes colored, every reduced trunk $R'$ in $\mathcal{R'}(F)$ either
\begin{enumerate}
\renewcommand{\labelenumi}{\roman{enumi})}
\setlength{\itemsep}{0pt}
\setlength{\parskip}{0pt}
\setlength{\parsep}{0pt}
\item has one colored vertex and does not have an edge in  $E_{>>2}(\mathcal{R}'(F))$ OR
\item has no colored vertices and there exists a vertex $v \in V(R')$ that covers $E_{>>2}(R')$.
\end{enumerate}
\end{theorem}
\begin{proof}
If a forest meets the condition stated in the theorem, it has game chromatic number at most 3 by Theorem  \ref{t:ra}. Otherwise, it has game chromatic number 4 by  Theorem \ref{ug1}.
\end{proof}

\section{A Tree $T$ with $\game(T)=4$ and $\Delta(T)=3$}

The motivation for this section arose from attempts to classify trees with game chromatic number 3.  While Theorem \ref{t:gcn2} completely classifies all trees and forests of game chromatic number 2, we still do not know what distinguishes trees and forests of game chromatic numbers 3 and 4.  

\begin{figure}[h]
\begin{picture}(400,130)
\put(42,36){\includegraphics[width=.38\linewidth]{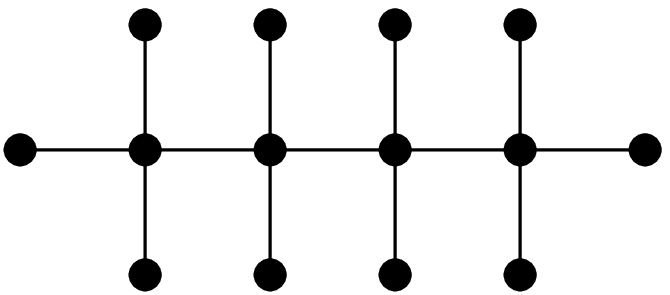}}
\put(213,7){\includegraphics[width=.28\linewidth]{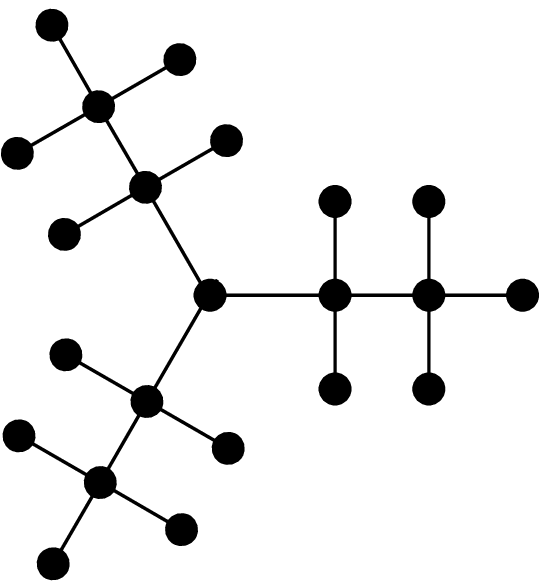}}
\end{picture}
\caption{The tree $T'$ is on the left.  The tree $T''$ is on the right.}\label{F.smallexample}
\end{figure}

\begin{figure}[h]
\leavevmode
\begin{center}
\includegraphics[width=.23\linewidth]{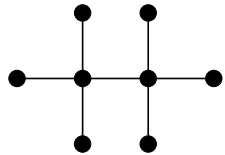}
\end{center}
\caption{The tree $H$.}\label{F.examplesubgraph}
\end{figure}

Thereom \ref{l:14v} and Lemma \ref{l:ecg} show that whenever $T'$ is a subgraph of a forest $F$ then $\game(F)=4$.  But this
is not a necessary condition.  The tree $T''$ in Figure \ref{F.smallexample} also has game chromatic number 4, but because the
center vertex has degree 3, we can see that $T'$ is not a subgraph of it.  We do, however, see
three copies of the subgraph $H$ (see Figure \ref{F.examplesubgraph})
in $T''$.  Bob exploits the subgraph $H$ as in Lemma \ref{l:pcl2} to win the 3-ECG, and thus the 3-coloring game, on $T''$.  We wanted to determine if maximum degree was a relevant characteristic for
classifying game chromatic number 3.  We conjectured that maximum degree alone is not
relevant; specifically, that there exists a tree with maximum degree 3 but game chromatic number 4.

In this section, we show that Bob can always win the 3-coloring game on a particular tree $T$.  This is sufficient to
prove that $\game(T)=4$.  When Bob is playing the 3-coloring game, we call coloring a vertex $v$ with $\alpha$ a \emph{winning move} if $v$ is uncolored, $\alpha$ is a legal move for $v$, and coloring $v$ with $\alpha$ will make some vertex uncolorable. We begin our proof that there exists a tree $T$ with maximum degree 3 but game chromatic number 4 with
a sequence of lemmas.  

\begin{lemma}\label{L:V1}
Let $T_1$ be the partially colored tree shown in Figure \ref{F.doubledanger} where $x$ can be colored with a color $\gamma$ different from both $c(u)$ and $c(v)$ and if $a$ or $b$ has a colored neighbor then it is colored $\gamma$.  If it is Bob's turn, then Bob can win the 3-coloring game.
\end{lemma}

\begin{figure}[h]
\leavevmode
\begin{center}
\includegraphics{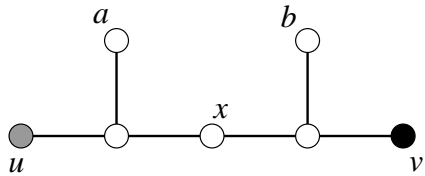}
\end{center}
\caption{Diagram showing $T_1$.  The colors $c(u)$ and $c(v)$ aren't necessarily distinct.}\label{F.doubledanger}
\end{figure}

\begin{proof}
Bob will color the vertex $x$ with $\gamma$, where $\gamma\neq c(u)$ and $\gamma\neq c(v)$.  Coloring $a$ and coloring $b$ with a carefully chosen color are now disjoint winning moves for Bob in the 3-coloring game.
\end{proof}

\begin{lemma}\label{L:V2}
If there is an uncolored $P_4$ subgraph in a tree $T$ that is surrounded by a color $\alpha$ (that is, each vertex on the $P_4$ has a neighbor colored $\alpha$ and is not adjacent to any color other than $\alpha$) and $|V(T)|$ is even, then Bob can win the 3-coloring game on $T$.
\end{lemma}

\begin{proof}
If Alice colors a vertex of the $P_4$ with $\beta\neq\alpha$, Bob can color a vertex of the $P_4$ at distance two away with a third color $\gamma$.  This leaves a vertex with no legal color available, so Bob wins if Alice plays first in the surrounded $P_4$.  If Bob ever has to color a neighbor of the surrounded $P_4$, he may safely color it with $\alpha$ so that the $P_4$ remains surrounded.  If Alice ever colors a neighbor of the surrounded $P_4$ some color other than $\alpha$ (so that it is no longer surrounded by $\alpha$) Bob can immediately play a winning move at distance two away from Alice's move.  Because $|V(T)|$ is even, Bob can force Alice to play first in the surrounded $P_4$, if he hasn't already won.  Thus Bob will win the 3-coloring game on $T$.
\end{proof}

\begin{lemma}\label{L:V3}
Let $T_2$ be the partially colored tree shown in Figure \ref{F.twoarms} where $c(u)=c(v)$ and $x,a$, and $b$ can be colored with $c(u)$.  If $T_2$ is a subgraph of $T$, $|V(T)|$ is even, and it is Bob's turn then Bob can win the 3-coloring game on $T$.
\end{lemma}

\begin{figure}[ht]
\leavevmode
\begin{center}
\includegraphics{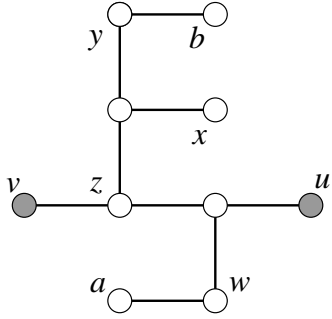}
\end{center}
\caption{Diagram showing $T_2$.  $c(u)=c(v)$ is an available color for $x,a,b$.}\label{F.twoarms}
\end{figure}

\begin{proof}
Bob will color $x$ with $c(u)$.  If Alice colors any neighbors of $u,v$, or $x$ or colors either $w$ or $y$ with any color not $c(u)$, then coloring a vertex at distance 2 away from Alice's move with a carefully chosen color is a winning move for Bob.  If Alice colors $w$ or $y$ with $c(u)$, then Bob can color $b$ or $a$ (respectively) with $c(u)$ to create a $P_4$ surrounded by $c(u)$ and win by Lemma \ref{L:V2}.  If Alice colors $a$ or $b$ with $c(u)$, then she has created a surrounded $P_4$ and Bob will win the 3-coloring game by Lemma \ref{L:V2}.  Alice's other possible moves are to color $a,b$ with a color other than $c(u)$, or to color a vertex outside of $T_2$.  In the former case Bob will color $z$ with $c(a)$ or $c(b)$ (depending on which Alice colored), and in the latter case Bob will color $z$ with any color not $c(u)$.  In both cases this is winning for Bob by Lemma \ref{L:V1} 
\end{proof}

In the proofs of the following Lemmas, Bob will frequently employ a strategy called a \emph{forcing move}.  This is where Bob colors a vertex $u$ at distance 2 from a colored vertex $v$ to create a winning move on the trunk containing $u$ and $v$.  Alice must respond in this trunk, or she will lose on Bob's next turn.  This allows Bob to use a sequence of forcing moves to produce a desired subgraph.

\begin{lemma}\label{L:V6}
Let $T_3$ be the partially colored tree shown in Figure \ref{F.uhoh2} where $c(u)=c(w)=c(x)$ and $c(v)=c(y)=c(z)$, but $c(u)\neq c(v)$. If $T_3\in\mathcal{R}(T)$ and $|V(T)|$ is even, then Bob can win the 3-coloring game on $T$.
\end{lemma}

\begin{figure}[ht]
\leavevmode
\begin{center}
\includegraphics[width=.85\linewidth]{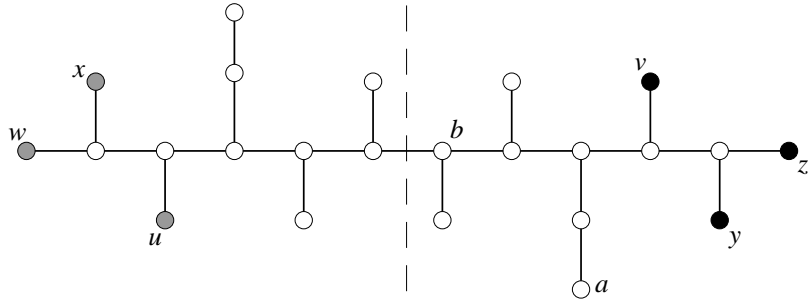}
\end{center}
\caption{Diagram showing $T_3$. By symmetry, we assume Alice colors a vertex to the right of the dotted line.}\label{F.uhoh2}
\end{figure}

\begin{proof}
Since $T_3$ is symmetric, we will only consider Alice's move to the right of the dotted line.  We label the vertices on $P_{u,v}$ with $x_i$ where $i=d(u,x_i)$ so $x_0=u$ and $x_9=v$.  If Alice colors the neighbor of $z$, then coloring $x_7$ with a carefully chosen color is a winning move for Bob.  If Alice colors $b=x_5$ then Bob wins by Lemma \ref{L:V1} since $d(b,v)=4$.  
If Alice colors $a$, the neighbor of $a$, or any vertex not in $T_3$, then Bob will color $x_7$ with an available color other than $c(v)$. Now Bob has a winning move in the trunk $R_{x_8}$, so his coloring of $x_7$ was a forcing move.  After Alice responds in the trunk $R_{x_8}$, Bob continues to play as if Alice's first move was to color $x_7$.  If Alice's first move was to color $x_7$ or any other vertex $c$ on the right side of $T_3$ (other than $b$, a case which is already covered), then Bob will be able to play a sequence of forcing moves in the direction of $u$ until Bob colors $x_4$ with any color or $x_3$ with $c(u)$.  Because it was a forcing move, Alice will have to respond to Bob's most recent move in a trunk other than $R_{x_1}$. Then Bob will color in the trunk $R_{x_1}$ so that he wins by Lemma \ref{L:V1} or so that he surrounds an uncolored $P_4$ with $c(u)$ and wins by Lemma \ref{L:V2}, depending on if he last colored $x_4$ or $x_3$.

Thus, Bob always wins the 3-coloring game on $T$ if $T_3\in\mathcal{R}(T)$ and $|V(T)|$ is even.
\end{proof}

\begin{lemma}\label{L:V7}
Let $T_4$ be the partially colored tree shown in Figure \ref{F.uhoh1} where $c(u)=c(w)=c(x)$ but $c(u)\neq c(v)$.  If $T_4\in\mathcal{R}(T)$ and $|V(T)|$ is even, then Bob can win the 3-coloring game on $T$.
\end{lemma}

\begin{figure}[ht]
\leavevmode
\begin{center}
\includegraphics[width=.85\linewidth]{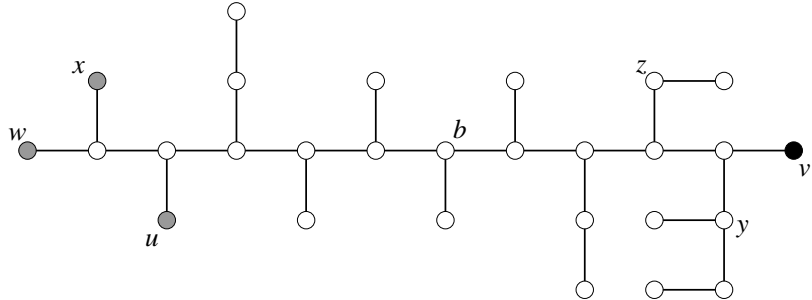}
\end{center}
\caption{Diagram of $T_4$.}\label{F.uhoh1}
\end{figure}

\begin{proof}
If Alice colors a vertex not in $T_3$ or one whose distance from the $u,v$-path is greater than 1, then Bob can play as if Alice just colored
$v$ and make forcing moves towards $u$, winning by Lemma \ref{L:V1}.  So suppose Alice colors a vertex whose distance is at most 1 from the 
$u,v$-path.  If Alice colors any vertex $a$ other than $b$ or $y$ then Bob can apply a sequence of forcing moves towards $u$ if $d(u,a)>5$ or towards $v$ if $d(v,a)>5$.  If Bob makes forcing moves towards $u$, he will win by Lemma \ref{L:V1} if $d(u,a)$ is even and by Lemma \ref{L:V2} if $d(u,a)$ is odd (Bob will be able to use $c(u)$ on his last forcing move and then surround an uncolored $P_4$ with $c(u)$).  If Bob makes forcing moves towards $v$, he will win by Lemma \ref{L:V1} or Lemma \ref{L:V3} depending on if $d(v,a)$ is even or odd (respectively).

Now we consider when Alice colors at $y$ or $b$.  If Alice colors $y$ then either Bob has an immediate winning move, or can reduce to Lemma \ref{L:V7} by coloring $z$ with $c(v)$. If Alice colors $b$, then either $c(b)\neq c(u)$ or $c(b)\neq c(v)$.  Bob will win by playing a forcing move towards $u$ with $c(u)$ or towards $v$ with $c(v)$, using whichever color is different from $c(b)$.  After Alice responds, Bob will win by Lemma \ref{L:V2} or Lemma \ref{L:V3}.
\end{proof}

\begin{lemma}\label{L:V8}
Let $T_5$ be the partially colored tree shown in Figure \ref{F.smallestsubgraph} where $c(u)\neq c(v)$.  If $T_5\in\mathcal{R}(T)$ and $|V(T)|$ is even, then Bob can win the 3-coloring game on $T$.
\end{lemma}

\begin{figure}[ht]
\leavevmode
\begin{center}
\includegraphics[width=.85\linewidth]{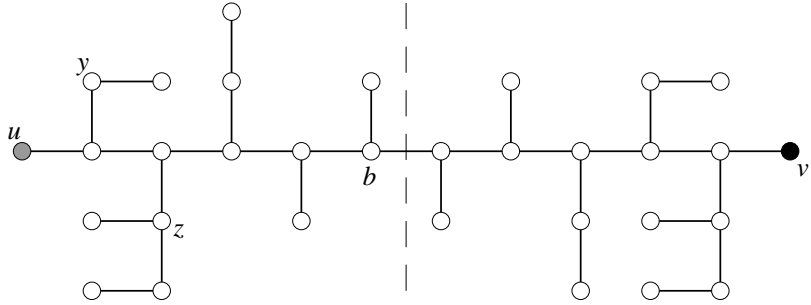}
\end{center}
\caption{Diagram of $T_5$.}\label{F.smallestsubgraph}
\end{figure}

\begin{proof}
If Alice colors a vertex not in $T_5$ or whose distance from the $u,v$-path is greater than 1,
then Bob can either play forcing moves from $u$ to $v$, or from $v$ towards $u$, depending on which side will let Bob win by Lemma \ref{L:V3} (Alice's move may prevent at most one side from meeting conditions for that Lemma).  So suppose Alice colors a vertex whose distance is at most 1 from the  $u,v$-path.  
The subgraph induced by these vertices is symmetric, so we will only consider when Alice's move is to the left of the dotted line.
If Alice colors some vertex $x\neq y$ then $d(v,x)>5$ and Bob can play forcing moves towards $v$ until he wins by Lemma \ref{L:V1} or Lemma \ref{L:V3}.  If Alice colors $y$, she must color it with $c(u)$ or Bob can make an immediate winning move.  If Alice colors $y$ with $c(u)$ then Bob colors $z$ with $c(u)$ and wins by Lemma \ref{L:V7}.
\end{proof}

We are now ready to prove our main result.

\begin{theorem}\label{T:V1}
There exists a tree $T$ with $\Delta(T)=3$ and $\game(T)=4$.
\end{theorem}

\begin{proof}
Consider three copies of $T_5$, we call them $T_5^1$, $T_5^2$, and $T_5^3$.  Let $T$ be the tree obtained from the forest $T_5^1\cup T_5^2\cup T_5^3$ by setting $u^1=u^2=u^3$ where $u^i\in V(T_5^i)$ is the vertex labelled $u$ in Figure \ref{F.smallestsubgraph}. Let $v^i\in T_5^i$ be the vertex labelled $v$ in Figure \ref{F.smallestsubgraph}. The tree $T$ has maximum degree of 3, and $|V(T)|=3(32)-2=94$ is even.  This construction mimics the way $T''$ (see Fig. \ref{F.smallexample}) is built using three copies of $H$ (see Fig. \ref{F.examplesubgraph}).

Bob and Alice play the 3-coloring game on $T$.  Without loss of generality, Alice's first move is on a vertex from $T_5^1$.  If Alice colors $u^1$, then Bob can color $v^1$ using a different color and win by Lemma \ref{L:V8}.  Otherwise, Bob colors $u^1=u^2=u^3$ with any available color.  If Alice's second move is in $T_5^1$ or $T_5^2$, Bob will color $v^3$ with a different color than $c(u^3)$.  If Alice's second move is in $T_5^3$, then Bob will color $v^2$ with a different color than $c(u^2)$.  In either case, Bob wins by Lemma \ref{L:V8}.

Because Bob can always win the 3-coloring game on $T$, we have shown that $\game(T)=4$.
\end{proof}

\section{Conclusions}
We provided criteria to determine the game chromatic number of a forest without vertices of degree 3. When vertices of degree 3 are permitted, the problem seemingly becomes more complex and no classification criteria are known. In fact, there is no known classification criteria for forests of maximum degree 3.

Throughout the work done in this paper, we never distinguished between vertices of degree 4 and vertices of degree more than 4. This surprising observation leads to the following conjecture.
\begin{conjecture}
If $F$ is a forest, there exists $F' \subseteq F$ such that $\Delta(F') \leq \game(F)$ and $\game(F')=\game(F)$.
\end{conjecture}

\begin{question}
Let $\mathcal{G}$ be the class of graphs such that for any $G \in \mathcal{G}$, there exists $G' \subseteq G$ such that $\Delta(G') \leq \game(G)$ and $\game(G')=\game(G)$.  What families of graphs are in $\mathcal{G}$?
\end{question}


\begin{thebibliography}{99}
\bibitem{B91}
      H.~Bodlaender, ``On the complexity of some coloring games,"
       \emph{Graph Theoretical Concepts in Computer Science}, (R. M\"{o}hring, ed.), vol. 484, 
       Lecture notes in Computer Science, Springer-Verlag, (1991) 
       30--40.
\bibitem{CZ}
    L.~Cai and X.~Zhu,
    ``Game chromatic index of $k$-degenerate graphs,''
    \emph{Journal of Graph Theory}, \textbf{36} (2001), no. 3, 
    144-155.
\bibitem{CWZ00}
   C.~Chou, W.~Wang, and X.~Zhu,
    ``Relaxed game chromatic number of graphs,''
    \emph{Discrete Mathematics}, \textbf{262} (2003) no. 1-3, 
    89--98.
\bibitem{DZ99}
    T.~Dinski and X.~Zhu,
    ``A bound for the game chromatic number of graphs,''
    \emph{Discrete Mathematics}, \textbf{196} (1999),
    109--115.
\bibitem{D07}
    C.~Dunn,
    ``The relaxed game chromatic index of $k$-degenerate graphs,''
    \emph{Discrete Mathematics}, \textbf{307} (2007),
    1767--1775.
\bibitem{DNNPSS11}
    C.\ Dunn, J.\ Firkins Nordstrom, C.\ Naymie, E.\ Pitney, W.\ Sehorn, C.\ Suer, 
    ``Clique-relaxed graph coloring,'' 
    \textit{Involve}, to appear.
\bibitem{DK1}
    C.~Dunn, H.A.~Kierstead, 
    ``A simple competitive graph coloring algorithm II,'' 
    \emph{Journal of Combinatorial Theory, Series B}, \textbf{90} (2004) 
    93--106.
\bibitem{DK2}
   C.~Dunn, H.A.~Kierstead, 
   ``A simple competitive graph coloring algorithm III,'' 
   \emph{Journal of Combinatorial Theory, Series B}, \textbf{92} (2004) 
   137--150.
\bibitem{DK3}
    C.~Dunn, H.A.~Kierstead, 
    ``The relaxed game chromatic number of outerplanar graphs,'' 
    \emph{Journal of Graph Theory}, \textbf{46} (2004) 
    69--78.
\bibitem{FKKT93}
    U.~Faigle, W.~Kern, H.A.~Kierstead, and W.~Trotter,
    ``On the game chromatic number of some classes of graphs,''
   \emph{Ars Combinatoria}, \textbf{35} (1993),
    143--150.
\bibitem{G81}
    M.~Gardner,
    ``Mathematical Games,''
    \emph{Scientific American}, (April, 1981) 23.
\bibitem{GZ99}
    D.~Guan and X.~Xhu,
    ``Game chromatic number of outerplanar graphs,''
    \emph{Journal of Graph Theory}, \textbf{30} (1999),
    67--70.
\bibitem{K00}
    H.A.~Kierstead,
    ``A simple competitive graph coloring algorithm,''
    \emph{Journal of Combinatorial Theory, Series B}, \textbf{78} (2000),
    57--68.
\bibitem{KT94}
    H.A.~Kierstead and W.~Trotter,
    ``Planar graph coloring with an uncooperative partner,''
    \emph{Journal of Graph Theory}, \textbf{18} (1994),
    569--584.
\bibitem{LSX99}
    P.~Lam, W.~Shiu, and B.~Xu,
    ``Edge game-coloring of graphs,''
    \emph{Graph Theory Notes of New York XXXVII}, (1999),
    17--19.
\bibitem{Z99}
    X.~Zhu,
    ``Game coloring of planar graphs,''
    \emph{Journal of Combinatorial Theory, Series B}, \textbf{75} (1999),
    245--258.
\bibitem{Z00}
    X.~Zhu,
    ``The game coloring number of pseudo partial $k$-trees,''
    \emph{Discrete Mathematics}, \textbf{215}~(2000), 
    245--262.
\end{thebibliography}
\end{document}